\documentclass[10pt]{amsart}
\usepackage{amsmath}
\usepackage{amssymb}
\usepackage{graphicx}
\usepackage{a4wide}
\usepackage{pifont}

\newcommand{\Z}{\ensuremath{\mathbb{Z}}}
\newcommand{\R}{\ensuremath{\mathbb{R}}}
\newcommand{\C}{\ensuremath{\mathbb{C}}}
\newcommand{\Q}{\ensuremath{\mathbb{Q}}}

\renewcommand{\S}{\ensuremath{{\mathcal S}}}
\newcommand{\M}{\ensuremath{{\mathcal M}}}
\newcommand{\Pp}{\ensuremath{{\mathcal P}}}
\newcommand{\Cc}{\ensuremath{{\mathcal C}}}
\newcommand{\Rr}{\ensuremath{{\mathcal R}}}

\newcommand{\aeq}{\Leftrightarrow}

\renewcommand{\rho}{\varrho}
\renewcommand{\epsilon}{\varepsilon}

\DeclareMathOperator{\id}{id}

\newtheorem{theorem}{Theorem}[section]
\newtheorem{lemma}[theorem]{Lemma}

\newtheorem{corollary}[theorem]{Corollary}
\newtheorem{definition}[theorem]{Definition}

\begin{document}

\title{Perfect colourings of cyclotomic integers}

\author{E.P. Bugarin}
\address{Fakult\"at f\"ur Mathematik, Universit\"at
  Bielefeld, 33501 Bielefeld, Germany}
\email{pbugarin@math.uni-bielefeld}

\author{M.L.A.N. de las Pe\~nas}
\address{Mathematics department, Ateneo de Manila University, Loyola
Heights, 1108 Quezon City, Philippines}
\email{mlp@math.admu.edu.ph}

\author{D. Frettl\"oh}
\address{Institut f\"ur Mathematik, FU Berlin, 14195 Berlin, Germany}
\email{dirk.frettloeh@udo.edu}

\maketitle

\begin{abstract}
Perfect colourings of the rings of cyclotomic integers with
class number one are studied. It is shown that all colourings induced
by ideals $(q)$ are chirally perfect, and vice versa. A necessary and
sufficient condition for a colouring to be perfect is obtained,
depending on the factorisation of $q$. This result yields the colour
symmetry group $H$ in general. Furthermore, the colour preserving group $K$ is
determined in all but finitely many cases. An application to
colourings of quasicrystals is given.
\end{abstract}

\section{Introduction}
\label{intro}
The study of colour symmetries of periodic patterns or point
lattices in two or three dimensions is a classical topic, see
\cite{gs} or \cite{schw}. A colour symmetry is a
symmetry of a coloured pattern up to permutation of colours; and
the study of colour symmetry groups is dedicated to the relation of
the colour symmetries of a coloured pattern to the
symmetries of the uncoloured pattern. During the last century, the
classification of colour symmetry groups of periodic patterns has been
carried out to a great extent. The discovery of quasiperiodic patterns
\cite{sen2} like the Penrose 
tiling raised the question about colour symmetries of these patterns.
Quasiperiodic patterns are not periodic, that is, the only translation
fixing the pattern is the trivial translation by 0. Nevertheless,
quasiperiodic patterns show a high degree of short and long range
order. One early approach to generalise the concept of colour symmetry
to quasiperiodic patterns was given in \cite{lip}. It used
the notion of indistinguishability of coloured patterns and the fact
that, for quasiperiodic patterns, it can be described in Fourier space
rather than in real space \cite{dm}. A more algebraic approach was
used in \cite{mp}, making use of quadratic number fields.
In this work, the problem of colour symmetries of both periodic and
non-periodic patterns, including the quasiperiodic cases, is addressed
by studying the sets of cyclotomic integers following the setting
introduced in \cite{b1}, \cite{bg}, \cite{bgs}. Cyclotomic integers 
turned out to
be very useful in describing symmetries of quasiperiodic patterns.

This article can be seen as a complement to
\cite{bg}, which concentrates on the combinatorial aspects of perfect
or chirally perfect colourings of
$\M_n$, where $\M_n = \Z[e^{2 \pi i / n}]$ denotes a $\Z$-module of
cyclotomic integers. (To be precise, they study so called `Bravais 
colourings', but we will show in the sequel that Bravais colourings and 
chirally perfect colourings are the same.)
In particular, the results in \cite{bg} yield the
numbers $\ell$ for which a (chirally) perfect colouring of $\M_n$ with
$\ell$ colours exists, given that $\M_n$ has class number one. In
contrast, this paper studies the algebraic properties of the colour
symmetry groups of perfect colourings of $\M_n$, again for the case
that $\M_n$ has class number one. In this sense this paper follows
the spirit of \cite{cbg}, where the group structure arising from of 
colourings (not necessariliy perfect ones) of periodic patterns are 
studied. Finally we should mention that the methods of this paper 
had been applied to a more detailed study of perfect colourings
of cyclotomic integers with $5$-fold, $8$-fold and $12$-fold
symmetry.

\section{Preliminaries}
\label{sec:1}
Let $\M_n$ denote the $n$-th ring of cyclotomic integers. That is,
$\M_n = \Z[\xi_n]$ is the ring of polynomials in $\xi_n$, where
$\xi_n=e^{2\pi i/n}$ always is a primitive complex root of unity. 
If it is clear from the context, we may write just $\xi$ instead of
$\xi_n$. 
Since $\M_{2n} = \M_n$ for $n$ odd, we omit the case $n \equiv 2 \mod
4$, for the sake of uniqueness. As mentioned above, our approach
requires that $\M_n$ has class number one. Then we can use the fact
that $\M_n$ is a principal ideal domain, and therefore also a unique
factorisation domain. This is only true for the following
values of $n$.
\begin{equation} \label{eq:cls1}
 n=3,4,5,7,8,9,11,12,13,15,16,17,19,20,21,24,25,27,28,32,33,
           35,36,40,44,45,48,60,84.
\end{equation}
Let us emphasise that $\M_n$ always denotes the ring of cyclotomic
integers for the values in Equation \eqref{eq:cls1} only.

{\bf Notation:} Throughout the text, $D_n$ (resp.\ $C_n$) denotes the
dihedral (resp.\ cyclic) group of order $2n$ (resp.\ $n$). The
symmetric group of order $n!$ is denoted by $\S_n$. Let $\xi_n =
e^{2 \pi i / n}$, a primitive $n$-th root of unity. The set of
cyclotomic integers $\Z[\xi_n]$ is denoted by $\M_n$.
The point group of $\M_n$ (the set of linear isometries fixing
$\M_n$) is the dihedral group
$D_N$, where $N=n$ if $n$ is even, and $N=2n$ if $n$ is odd.
The entire symmetry group $G(\M_n)$ of $\M_n$ is {\em symmorphic},
that is, it equals the semidirect product of its translation subgroup
with its point group: $G(\M_n) = \M_n \rtimes D_N$, where $N=n$ if $n$
is even, and $N=2n$ if $n$ is odd. If $H$ is a subgroup of some group
$G$, the index of $H$ in $G$ is denoted by $[G:H]$. Throughout the
text we will identify the Euclidean plane with the complex plane.
The complex norm of $z \in \C$ is always denoted by $|z|$, while the
algebraic norm of $z \in \M_n$ is denoted by $N_n(z)$.

The symmetry group of some set $X \subset \R^2$ is always
denoted by $G$ in the sequel. The following definitions are mainly
taken from \cite{gs}. 
A {\em colouring} of $X$ is a surjective map  $c: X \to \{1, \ldots,
\ell\}$. Whenever we want to emphasise that a
colouring uses $\ell$ colours, we will also call it an $\ell$-colouring.
The objects of interest are colourings where an element of $G$ acts as
a global permutation of the colours. Thus, for given $X 
\subset \R^2$ and a colouring $c$ of $X$, we consider the following group.
\begin{equation} 
 H = \{ h \in G \, | \, \exists \, \pi \in \S_{\ell} \;  \forall x \in
 X: \; c(h(x)) = \pi (c(x))\}. 
\end{equation}
The elements of $H$ are called {\em colour symmetries} of $X$. $H$ is
the {\em colour symmetry group} of the coloured pattern $(X,c)$. 
\begin{definition} \label{def:perfect}
A colouring $c$ of a point set $X$ is called {\em perfect}, if
$H=G$. It is called {\em chirally perfect}, if $H=G'$, where $G'$ is
the index 2 subgroup of $G$ containing the orientation preserving
isometries in $G$. 
\end{definition}
See Figure \ref{fig:bspcol} for some examples. 
By the requirement $\pi c = c h$, each $h$ determines a unique
permutation $\pi = \pi_h$. This also defines a map
\begin{equation} \label {eq:p}
P: H \to \S_{\ell}, \quad P(h):=\pi_h.
\end{equation}
Let $g,h\in H$. Because of $c (hg(x)) = ch(g(x)) = \pi_h c(g(x)) =
\pi_h (\pi_g (c(x))) = \pi_h \pi_g (c(x))$, we obtain the following
result. 
\begin{lemma}
$P$ is a group homomorphism. \hfill$\square$
\end{lemma}
A further object of interest is the subgroup $K$ of $H$ which fixes
the colours. For a given $X \subset \R^2$ and a colouring $c$ of $X$,
we consider the {\em colour preserving group} $K$ (in \cite{pbeff}
called colour {\em fixing} group): 
\begin{equation} 
K := \{ k \in H \, | \, c(k(x)) = c(x), \, x \in X \}.
\end{equation}
In other words, $K$ is the kernel of $P$.
The aim of this paper is to deduce the nature of the groups $H$ and
$K$ for (chirally) perfect colourings of $\M_n$.

\begin{figure}
\includegraphics[width=120mm]{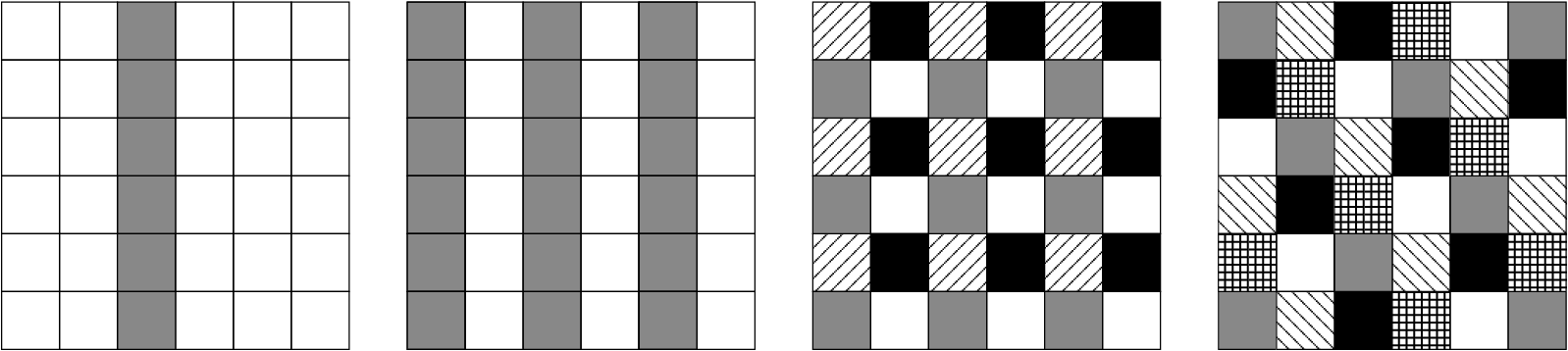}
\caption{Four examples of colourings of $\Z^2$. For clarity, each
  element of $\Z^2$ is replaced by a unit square. From left to right:
  An arbitrary 2-colouring with two colours, neither ideal nor perfect;
  a 2-colouring induced by a coset colouring, but neither ideal nor
  perfect; a perfect 4-colouring induced by the ideal $(2)$; a chirally
  perfect 5-colouring induced by the ideal $(2+i)$. \label{fig:bspcol}
}
\end{figure}

\section{Coset colourings and ideal colourings of planar modules}
\label{sec:2}
A colouring of a point set $X$ with a group structure (like a
lattice or a $\Z$-module) can be constructed by choosing a subgroup of
$X$ and assigning to each coset a different colour (\cite{vdw}, see
also \cite{mlp}).
Thus we will generate colourings of $\M_n$ by suitable
subgroups of $\M_n$. Since $\M_n$ is in fact a
principal ideal domain, we will choose principal ideals $(q)$ as these
subgroups. Each element $q \in \M_n$ thus generates a colouring in the
following way.

\begin{definition}
An {\em ideal colouring} of $\M_n$ with $\ell$ colours is defined as follows:
For each $z \in (q) = q \M_n$, let $c(z)=1$. Let
the other cosets of $(q)$ be $(q)+t_2, \ldots, (q) + t_{\ell}$.
For each $z \in (q)+t_i$, let $c(z)=i$.
\end{definition}

If $(q)$ is given explicitly, we will also call such an ideal colouring
a {\em colouring induced by} $(q)$.
We will see that all chirally perfect colourings of $\M_n$,
where $n$ is of class number one, arise from principal ideals $(q) = q
\M_n$, where $q \in \M_n$. Consequently, there exists a chirally
perfect colouring of $\M_n$ with $\ell$ colours, if and only if there
is $q$ such that $N_n(q) = [ \M_n : (q)] = \ell$. (Note that the index
of $(q)$ in $\M_n$ is just the algebraic norm of $q$.)

In \cite{bg}, the number of Bravais colourings of $\M_n$ was
obtained for all $n$ as in \eqref{eq:cls1}. Let us shortly explain, why
in this context Bravais colourings are chirally perfect colourings, and
vice versa. A {\em Bravais colouring} of $\M_n$ is a colouring where
each one-coloured subset is in the same
Bravais class as $\M_n$. In plain words, this means that each
one-coloured subset is similar to $\M_n$. More precisely: there is
$q \in \C$ such that for each $i$, $c^{-1}(i)$ is a translate of
$q \M_n$. For a general definition of Bravais class, see for instance
\cite{m}.

\begin{theorem}
Let $\M_n = \Z[\xi_n]$ be a principal ideal domain. A colouring of
$\M_n$  is a Bravais colouring, if and only if it is a chirally
perfect colouring, if and only if it is an ideal colouring.
\end{theorem}
\begin{proof}
Let $c$ be an ideal colouring induced by $(q)$.
Trivially, $(q) = q \M_n$ is similar to $\M_n$, and the cosets are
translates of $q \M_n$. Thus $c$ is a Bravais colouring.

Let $c$ be a Bravais colouring of $\M_n$. Without loss of generality,
let $0 \in c^{-1}(1)$. The set $c^{-1}(1)$ of points of colour 1 is
similar to $\M_n$, that is, it equals $q \M_n$ for some $q
\in \C$. Since $c^{-1}(1) \subset \M_n$, we have $q \M_n \subset
\M_n$, which implies $q \in \M_n$.  Thus $c^{-1}(1) = (q)$. All other
preimages $c^{-1}(i)$ are translates of $(q)$, thus cosets of $(q)$ in
$\M_n$. Therefore $c$ is an ideal colouring.

For the equivalence of chirally perfect colouring and ideal colouring,
see Theorem \ref{thm:bal} below.
\end{proof}

\section{The structure of $H$}
\label{sec:3}

Recall that $P: H \to \S_{\ell}$ maps a colour symmetry to the
permutation it induces on the colours, see \eqref{eq:p}.

\begin{lemma} \label{lem:h/k}
$H$ acts transitively on the coloured subsets of any perfect
colouring of $\M_n$, and $H/K \cong P(H)$.
\end{lemma}
\begin{proof}
The proof of the first statement follows from the proof of Theorem
\ref{thm:bal} below, see the remark there.
Since $K = \mbox{ker}(P)$, the second claim is clear.
\end{proof}
This yields the short exact sequence
\begin{equation} \label{eq:exseq}
 0 \longrightarrow K \longrightarrow H \longrightarrow
 H/K \longrightarrow 0.
\end{equation}
Therefore, $H$ is always a group extension of $K$. In general, $H$ is
neither a direct nor a semidirect product of $K$ and $H/K$, see
Theorem \ref{thm:prod} below.

We proceed by examining how the factorisation of $q$ in $\M_n$ affects
the structure of the colour symmetry group $H$ of the colouring
induced by $(q)$. The unique factorisation of $q$ over $\M_n$ reads
\begin{equation} \label{eq:qfac}
 q = \varepsilon \prod_{p_i \in \Pp} p_i^{\alpha_i} \prod_{p_j \in \Cc}
 \omega_{p_j}^{\beta_j}  \overline{\omega_{p_j}}^{\gamma_j}
 \prod_{p_k \in \Rr} p_k^{\delta_k},
\end{equation}
where $\varepsilon$ is a unit in $\M_n$.
Here, $\Pp$ (resp.\ $\Cc$, resp.\ $\Rr$) denotes the set of inert (resp.\
complex splitting, resp.\ ramified) primes over $\M_n$. The generator
$q$ is called {\em balanced} if $\beta_j=\gamma_j$ for all $j$. In
other words: $q$ is balanced if it is of the form
\begin{equation} 
 q= \epsilon x p,
\end{equation}
where $\epsilon$ is a unit in $\M_n$, $x$ is a real number in $\M_n$
(i.e., $x \in \Z[\xi+\overline{\xi}]$), and $p$ is a
product of ramified primes. By the definition of a ramified prime $p$
(see \cite{wash}), $\overline{p} \in (p)$ holds in
$\M_n$. (Equivalently, $p / \overline{p}$ is a unit in $\M_n$.)
The following lemma is well-known, it is stated here for the convenience of
the reader.

\begin{lemma} \label{lem:unit}
All units $\epsilon$ in $\Z[\xi_n]$ are of the form $\epsilon=\pm \lambda
\xi_n^k$, where $\lambda \in \Z[\xi+\overline{\xi}]$.
\end{lemma}
\begin{proof}
(Essentially \cite{wash}, Prop. 1.5:) Let $\epsilon$ be a unit, and let
$\alpha=\epsilon/
\overline{\epsilon}$. Since $\epsilon, \overline{\epsilon},
1/\overline{\epsilon} \in \M_n$, $\alpha$ is an algebraic integer.
Since complex conjugation commutes with any element of the Galois
group, for all algebraic conjugates $\alpha_i$ of $\alpha$ holds
$|\alpha_i|=1$.

From Lemma 1.6 of \cite{wash} then follows:
If for all algebraic conjugates $\alpha_i$
of $\alpha$ holds $|\alpha_i|=1$, and $\alpha$ is an algebraic
integer, then $\alpha$ is some root of unity, say, $\xi_r^q$. Since
$\alpha \in \M_n$, it is either an $n$-th root of unity, or a $2n$-th
root of unity, if $n$ is odd. In each case, $\alpha=\pm \xi_n^j$ for
some $j$. Then $\epsilon^2=\epsilon \overline{\epsilon} \alpha = |
\epsilon | \xi_n^j$, where $| \epsilon|$ is a real number, thus $|
\epsilon | \in \Z[\xi_n + \overline{\xi}_n]$. It follows
$ \epsilon = \sqrt{|\epsilon|} \xi^{j/2}_n = \pm \sqrt{|\epsilon|}
\xi^{j}_{2n} = \pm \lambda \xi^{j}_{2n}$, where $\lambda \in  \Z[\xi_n
+ \overline{\xi}_n]$. However, since $\epsilon \in \Z[\xi_n]$, it
can't be a proper $2n$-th complex root of unity. Thus  $\pm \lambda
\xi^{j}_{2n} =  \pm \lambda \xi^{k}_n$ for some $k$.
\end{proof}
In particular, if $\epsilon$ is a unit in $\Z[\xi_n]$ with $|\epsilon|
= 1$, then it is (up to sign) an $n$-th complex root of unity:
$\epsilon = \pm \xi_n^k$.
\begin{lemma} \label{lem:bal}
$\overline{q} \in (q)$ if and only if $q$ is balanced.
\end{lemma}
\begin{proof}
Let $q$ be as in \eqref{eq:qfac}. Consider $q/\overline{q}$. The inert
primes in numerator and denominator cancel each other. The
unit $\varepsilon$, as well as the factors of the ramified primes,
contribute a unit $\varepsilon' \in \M_n$. Thus
\[ r:= q/\overline{q} =  \varepsilon' \prod_{p_j \in \Cc}
\omega_{p_j}^{\beta_j-\gamma_j}
\overline{\omega_{p_j}}^{\gamma_j-\beta_j}. \]
If $q$ is balanced, then $\beta_j=\gamma_j$, thus $\overline{q} =
(\varepsilon')^{-1} q \in (q)$, since $\varepsilon'$ is a unit. If $q$
is not balanced, then $\beta_j-\gamma_j \ne 0$ for some $j$. Then, by
Lemma \ref{lem:unit}, the right hand side $r$ is not a unit, thus
$r^{-1} \notin \M_n$, and consequently $\overline{q} = r^{-1} q \notin
(q)$. 
\end{proof}

\begin{theorem} \label{thm:bal}
Let $\M_n=\Z[\xi_n]$ be a principal ideal domain.
\begin{enumerate}
\item Each chirally perfect colouring of $\M_n$ is an ideal colouring.
\item Each ideal colouring of $\M_n$ is chirally perfect.
\item The colouring $c$ induced by $(q)$ is perfect, if and only if $q$
  is balanced.
\end{enumerate}
Consequently, $H= \M_n \rtimes D_N$ in the latter case, and $H= \M_n
\rtimes C_N$ otherwise. ($N=2n$ if $n$ is odd, $N=n$ else.)
\end{theorem}

\begin{proof}
Let $(q)$ be the ideal inducing the colouring of $\M_n$. Let
$\ell=[\M_n:(q)]$, and denote the cosets of $(q)$ by $(q)+t_1, \ldots,
(q)+t_{\ell}$, where $t_1=0$ for convenience. (The notation
$(q)+g$ rather than $g(q)$ or $(q) g$ is justified as follows:
all rotations and reflections in $G$ fix $(q)$. Only maps with
some translational part map $(q)$ to a coset different from $(q)$.)

We proceed by studying whether an element $g \in G$ maps an entire
coset $(q)+t_i$ to an entire coset $(q)+t_j$ or not. If yes, then $g$
induces a global permutation of the colours, and $g \in H$. Three
cases have to be considered.

1. Let $g \in G$ be a translation. Then $g$ is of the form
$g(x)=x+t$ for some $t$. Hence $g((q)+t_i)=(q)+t_i+t$ trivially is a
coset of $(q)$.

2. Let $g \in G$ be a rotation. Then $g((q))=(q)$, thus $g((q)+t_j)=(q) +
g(t_i)$, which is again a coset of $(q)$.

So, the first two cases are not critical, whether $q$ is balanced or
not. In particular, all orientation preserving isometries map entire cosets to entire
cosets, which proves part (2) of Theorem \ref{thm:bal}.

3. Let $g \in G$ be the reflection $x \mapsto \overline{x}$. If $q$ is
   balanced, then, by Lemma \ref{lem:bal}, $g(q) = \overline{q} \in
   (q)$, hence $g((q))=(q)$. Thus $g((q)+t_j) = (q) + g(t_j)$, which is a
   coset of $(q)$. Any element of $G$ is a composition of the three
   symmetries above, hence the `if'-part of Theorem \ref{thm:bal} (3) follows.

If $q$ is not balanced, then $g(0)=0 \in (q)$, but $g(q) \notin (q)$ by
Lemma \ref{lem:bal}. Thus $g$ does not map entire cosets to entire
cosets. Consequently, no reflection in $G$ maps entire cosets to
entire cosets. The reflections in $G$ are the only elements which fail
to do so. This (again) shows Theorem \ref{thm:bal} (2), and the `only-if' 
part of Theorem \ref{thm:bal} (3).

Regarding Theorem \ref{thm:bal} (1): Let $c$ be a chirally perfect
colouring. Let $0 \in c^{-1}(1)$. Then $c^{-1}(1)$ is invariant and
closed under rotations in $C_N$, and under translations by $t \in
c^{-1}(1)$. It follows that $c^{-1}(1)$ is invariant under
multiplication by elements of $\M_n$, and under translations by
elements of $c^{-1}(1)$. Thus $c^{-1}(1)$ is an ideal in $\M_n$.
\end{proof}
This proves Lemma \ref{lem:h/k} as well: $H$ contains all translations
$z \mapsto z+t, \; t \in \M_n$. Clearly, these translations act
transitively on the cosets.

\begin{corollary} \label{cor:one}
If there exists only one ideal colouring of $\M_n$ with $\ell$
colours, the colouring is perfect.
\end{corollary}
\begin{proof}
If an ideal colouring induced by $(q)$ is not perfect, then $q$ is not
balanced by Theorem \ref{thm:bal}. Thus, $\overline{q} \notin (q)$ by
Lemma \ref{lem:bal}, hence $(q) \ne (\overline{q})$. So $(q)$ and
$(\overline{q})$ define two different colourings with $\ell$ colours.
\end{proof}
Now we get immediately a result on colourings of $\Z^2$. This is
Theorem 8.7.1 in \cite{gs}, see also \cite{sen}. Note that the number of colours $\ell$ is
the norm $N_4(q)$ of $q$, which is just $q\overline{q}$.
\begin{corollary} \label{cor:z2}
Let $c$ be an $\ell$--colouring of the square lattice $\Z[i]$ generated by
$(q)$,  $q \in \Z[i]$.
\begin{enumerate}
\item If the factorisation of $\ell$ over $\Z$ contains no primes
  $p \equiv 1 \mod 4$, then the colouring is perfect.
\item If $q=m$, or $q=im$, or $q=(1 \pm i)^k m$ for some $m, k \in \Z
\setminus \{0\}$, then the colouring is perfect.
\item Otherwise the colouring is not perfect but chirally perfect, and
  so $H = \M_n \rtimes C_N$. 
\end{enumerate}
\end{corollary}
\begin{proof}
The inert primes in $\M_4=\Z[i]$ are exactly the ones of the form $p
\equiv 3 \mod 4$; and the splitting primes are exactly those of the
form $p \equiv 1 \mod 4$. The only ramified prime in $\M_4$ is
$2=(1+i)(1-i)$. So (1) and (2) of Corollary \ref{cor:z2} cover exactly
the cases where 
$q$ is balanced, and the claim follows from Theorem \ref{thm:bal}.
\end{proof}

Corollary \ref{cor:z2} tells us that all ideal colourings of the
square lattice with $1,2,4,8,9,16$ or $18$ colours are perfect, and
all those with $5,10,13,17,20$ colours are not. (These are all
possible values for $\ell<25$, see \cite{bg}). The first ambiguity
occurs at the value $\ell=25$: the three possible generators are
$q=5, \; q=3+4i,\; q=3-4i$. The first one induces a perfect
colouring, whereas the other two induce non-perfect but chirally
perfect colourings.

\begin{theorem} \label{thm:prod}
If $\ell = 2$, then $H$ is a semidirect product: $H = K \rtimes H/K$.
If $\ell > N$, then $H \ne K \rtimes H/K$. ($N=2n$, if $n$ odd, $N=n$
else.)
\end{theorem}
\begin{proof}
Consider Equation \eqref{eq:exseq}. By the splitting lemma, $H = K
\rtimes H/K$, if and only if there is a homomorphism $Q: H/K \to H$
such that $PQ = \id$ on $H/K$. If $\ell=2$, then $H/K \cong C_2 =
\{ \id, x \}$.  Let $Q(\id)=\id$ and $Q(x) = \varphi$, where $\varphi$
is the reflection in the vertical line through $\frac{1}{2}$. Certainly,
$\varphi \in G$ holds: $\varphi: a + bi \mapsto 1-a +bi$ is a
composition of  $z \mapsto \overline{z}, \; z \mapsto iz, \; z \mapsto
z+1$.  By Lemma \ref{lgleich2} (see below), the colouring is perfect,
thus $f \in H$, and $c(0) \ne c(1)$, thus $f$ interchanges the two
colours. This makes $Q$ a homomorphism.

In general, there is no such homomorphism $Q$: All elements
$\pi \in H/K$ are of finite order. If $Q(\pi)$ contains a translational
part, it is of infinite order in $H$. Thus there is $k = \mbox{ord}(\pi)$
such that $\id = Q(\id) = Q(\pi^k) \ne Q(\pi)^k$, hence $Q$ is not a
homomorphism.

The elements $z \in \M_n$ with $|z|=1$ are exactly the $N$ elements
of the form $\pm \xi^i_n$ (see Lemma \ref{lem:unit}). Thus they can
carry at most $N$
colours. They can be mapped to each other by rotations about 0, or
by reflection $z \mapsto \overline{z}$. $H$ acts transitively on the
colours. Thus in any colouring with more than $N$ colours there has to
be a map $g \in H$ which is neither a rotation about 0, nor a
reflection $z \mapsto \overline{z}$. Thus the colouring requires a map
$h$ with some translational part, which is of infinite
order. Consequently, $Q(h)$ is of infinite order.
\end{proof}

The results in this section yield $H$ in general --- that is, whether
a colouring of $\M_n$ is perfect or not --- depending on the
factorisation of the generator of the underlying ideal. In the next
section we obtain results yielding $K$, depending only on the number
$\ell$ of colours. As a byproduct, we also obtain partial results on $H$,
depending on $\ell$ only.

\section{The structure of $K$}\label{sec:4}
This section contains several lemmas which determine the structure of 
$K$ in all but finitely many cases. Recall that $\ell$ denotes the 
number of colours and $c(x)$ denotes the colour of $x$. We denote the group
of translations by elements in $(q)$ by $T_{(q)}$. Note that $T_{(q)}$
is always contained in $K$.

\begin{lemma} \label{0ungleichxi}
$\ell \ge 2$ if and only if $c(0) \ne c(\pm\xi^i)$ for all $i \le n$.
\end{lemma}
\begin{proof}
$\ell = 1 \aeq (q)=(1) \aeq (q)=(\pm\xi^i)$ for some $i \aeq 0, \xi^i \in
  (q) \aeq c(0)=c(\xi^i)$.
\end{proof}

Let $\phi$ denote Euler's totient function.

\begin{lemma} \label{lgleich2}
Each 2-colouring of $\M_n$ induced by $(q)$ is perfect. Moreover,
$\ell = 2$, if and only if $c(\pm\xi^i) = c(\xi^j)\;
\mbox{for all}\; i,j \le n$, if and only if $K=T_{(q)} \rtimes D_N$.
\end{lemma}
\begin{proof}
Since $\ell=2$, we consider two cosets of $(q)$, namely, $(q)$ and
$(q)+1$. Note that $2 \in (q)$ and $\pm\xi^i\in (q)+1$ for any $i \in
\Z$. Consequently $\pm\xi^i \pm\xi^j \in (q)$, while $\pm\xi^i \pm\xi^j
\pm\xi^k \in (q)+1$ for any $i,j,k\in\Z$, and in general:
If $z=\sum_{i=0}^{\phi(n)-1} \alpha_i\xi^i \in \M_n$, with $\alpha_i
\in \Z$, then $z \in (q)$ if and only if
$\sum_{i=0}^{\phi(n)-1} \alpha_i \equiv 0 \mod 2$, otherwise $z \in (q)+1$. Now,
$\overline{z} = \sum_{i=1}^{\phi(n)-1} \alpha_i\xi^{n-i}$, and by conjugation of
$z$, the sum $\sum_{i=0}^{\phi(n)-1} \alpha_i$ does not change.
This implies that ${z} \in (q)$ if and only if $\overline{z} \in
(q)$. Similarly, $z \in (q) + 1$ if and only if $\overline{z} \in (q)
+ 1$. Thus the reflection
$z \mapsto \overline{z}$ maps $(q)$ to itself and $(q) + 1$ to
$(q)+1$. Hence the reflection is in $H$, and so the colouring is
perfect. Furthermore it fixes the coloured pattern, and so is also in
$K$.

From Lemma \ref{0ungleichxi} follows that $c(\pm\xi^i) \ne c(0) \ne
c(\xi^j)$ for all $i,j$. Since $\ell = 2$, that is, there are two colours
only, it follows $c(\pm\xi^i) = c(\xi^j)$.

Vice versa, if $c(\pm\xi^i)
= c(\xi^j)$ for all $i,j$, then $\pm\xi^i-\xi^j, 2 \in (q)$. This
means, analogous to the reasoning above,
$(q) = \big\{\sum_{i=0}^{\phi(n)-1} \alpha_i \xi^i \, | \,
\sum_{i=0}^{\phi(n)-1}\alpha_i\equiv 0\mod2, \; \alpha_i \in \Z \big\}$,
and thus $(q)$ has only one other coset, say $(q)+1$. This settles the first
equivalence.
Since $\pm \xi^i (q) =(q)$, $\pm \xi^i ((q)+1) = (q) \pm \xi^i = (q)+1$, it
follows that both the cosets are invariant under $N$-fold rotations, and so
$K=T_{(q)} \rtimes D_N$, since the reflection is also in $K$ as noted
above. Vice versa, if $K=T_{(q)} \rtimes D_N$, then
$c(\pm\xi^i) = c(\xi^j)$.
\end{proof}

\begin{lemma} \label{elllarge}
For all $\ell$-colourings of $\M_n$ holds: If $\ell > 2^{\phi(n)}$ then
$K=T_{(q)}$.
\end{lemma}
\begin{proof}
Recall that $K$ is a subgroup of $T_{(q)} \rtimes D_N$.
If some $\id \ne g \in D_N$ is an element of $K$, then $g$ maps some
$\pm \xi^i$ to some $\xi^j \ne \pm\xi^i$, with $c(\xi^i) = c(\pm\xi^j)$.
Thus it suffices to show $c(\xi^i) \ne c(\pm \xi^j)$ for all $i \ne
j$.

Assume $c(\xi^i) = c(\pm \xi^j)$. Then $\xi^i \pm \xi^j \in (q)$. Hence
\[ N_n(\xi^i \pm \xi^j) = \bigg| \prod_{k=1}^{\phi(n)} \sigma_k(\xi^i \pm
\xi^j) \bigg| = \prod_{k=1}^{\phi(n)} | \xi^{i_k} \pm \xi^{j_k} | \le
2^{\phi(n)}, \; \mbox{where} \; \sigma_k \in Gal(\Q(\xi_n),\Q). \]
It follows $\ell = [\M_n : (q)] \le  N_n(\xi^i \pm \xi^j) \le
2^{\phi(n)}$, which contradicts $\ell > 2^{\phi(n)}$.
\end{proof}

\begin{lemma} \label{lem:q=2}
If $(q)=(2)$, then $H=G$. Furthermore, for all $n \ne 4$ in
\eqref{eq:cls1}: $K=T_{(q)} \rtimes C_2$; and for $n=4$: $K=T_{(q)}
\rtimes D_2$. 
\end{lemma}
\begin{proof}
$q=2$ is balanced. Thus, $H=G$.

Now for $K$, compare the last proof: Note that $N_n(2)=2^{\phi(n)}$, 
and $c(1)=c(-1)$. But for all other complex roots of unity $\pm \xi_n^k$
($0<k<n$) holds $c(1)\ne c(\xi_n^k)$, since otherwise
we would have a factor of modulus strictly less than 2 in the
equation above, and the $\le$ becomes $<$.  Consequently,
$c(\xi^i)=c(-\xi^i)$, and so the only rotation in $K$ is the
rotation by $\pi$ about 0.

The reflection $z \mapsto \overline{z}$ maps $\xi$ to $\xi^{n-1}$. For
all $n \ne 4$ in \eqref{eq:cls1}, $N_n(\xi-\xi^{n-1}) < 2^{\phi(n)}$
and so $c(\xi)\ne c(\xi^{n-1})$. Thus the reflection is not contained
in $K$, and so $K=T_{(q)} \rtimes C_2$. Only for the case $n=4$ we get
$N_n(\xi-\xi^{n-1}) = 2^{\phi(n)}$, and so the reflection is in $K$
and thus $K=T_{(q)} \rtimes D_2$. 
\end{proof}

Why is the case $n=4$, $\ell = 2^{\phi(4)} = 4$ different? By inspection
of this case (see Figure \ref{fig:bspcol}) we find that $K = T_{(2)}
\rtimes D_2$. This is because $c(1) = c(-1)$ and $c(i) = c(-i)$; and only
in this case does the reflection $z \mapsto \overline{z}$ also belong in $K$.

\begin{lemma} \label{lem:normq=2n}
 If $\ell = 2^{\phi(n)}$ but $(q)\neq(2)$, then $K=T_{(q)}$
\end{lemma}
\begin{proof}
From the proof of the previous lemma, there can be no more symmetries
than the rotation by $\pi$ that can fix the colours. If this rotation
by $\pi$ indeed fixes the colours, then in
particular $c(1)=c(-1)$. Thus $2 \in (q)$, and so
$(2) \subseteq (q)$. But since $(q)$ and $(2)$ have equal algebraic
norms, then it follows that $(q)=(2)$, which is a contradiction.
Therefore, $K=T_{(q)}$.
\end{proof}

The case $\ell = 2^{\phi(n)}$ but $(q)\neq(2)$ first occurs when
$n=7$, see Table \ref{tab}.

\begin{lemma} \label{lem:lnprim}
If $2 < \ell = n$, where $n$ is prime in $\Z$, then $H=G$ and $K=T_{(q)}
\rtimes D_n$.
\end{lemma}
\begin{proof}
Note that in the case when $n$ is an odd prime, the symmetry group of
$(q)$ contains $D_N = D_{2n}$.

Let $2 < \ell = n$ and $n$ prime in $\Z$. Then the unique
factorisation of $\ell = n$ in $\M_n$ is $\ell = \prod_{i=1}^{n-1}
(1-\xi^i)$ \cite{wash}. Thus $\ell$ ramifies, and the possible
generators of the ideal $(q)$ are exactly the $1-\xi^i$. Therefore, by
Theorem \ref{thm:prod}, each corresponding colouring is perfect. In fact,
there is only one such colouring, since for all $1 
\le j \le n$ holds: $1-\xi^j \in (1-\xi)$. (This follows from
$\xi^k(1-\xi) \in (q)$, thus $\sum_{k=0}^{j-1} \xi^k(1-\xi) = 1-\xi^j
\in (q)$.) Moreover, it follows  that $c(1)=c(\xi^j)$ for all $j$.

Since $\ell$ is prime in $\Z$, we have $\M_n / (q)\cong C_\ell$, and
so the $\ell$ distinct cosets can be expressed as $(q), (q) + 1, (q) +
2, \ldots, (q) + \ell - 1$. Each coset is invariant under
multiplication by $\xi^j$, but not under multiplication by $-\xi$.
Thus, $K=T_{(q)} \rtimes D_n$.
\end{proof}

\begin{lemma} \label{krefl}
If $H=G$ and $\ell$ is prime in $\Z$, then $K$ contains
a reflection. Thus, $T_{(q)} \rtimes C_2$ is a subgroup of $K$.
\end{lemma}
\begin{proof}
Because $\ell$ is prime, the $\ell$ distinct
cosets are $(q), (q) + 1, \ldots, (q) + \ell -
1$. Clearly, these cosets are invariant under conjugation, hence the
reflection $z \mapsto \overline{z}$ is contained in $K$. Consequently,
$K$ contains $T_{(q)} \rtimes C_2$ as a subgroup.
\end{proof}

The previous lemma together with Lemma \ref{elllarge} yields
the following result immediately.

\begin{lemma} If $\ell$ is prime in $\Z$ and $K = T_{(q)}$, then
  $H=G'$. In particular, if $\ell>2^{\phi(n)}$ is prime in $\Z$, then
  $H=G'$. \hfill  $\square$
\end{lemma}

\begin{lemma}
If $\ell > n$ and $\ell$ is prime in $\Z$, then $H=G'$. 
\end{lemma}
\begin{proof}
If $H=G$, then by Lemma \ref{krefl} the
cosets must be fixed by taking conjugates. This would mean that
$c(\xi^i)=c(\xi^{n-i})$ and so $\xi^i(1-\xi^{n-2i})\in(q)$. Thus $\ell =
N_n(q)\mid N_n(1-\xi^{n-2i})=: \alpha$. Now, $\alpha$ must be either
$2^{\phi(n)}$ (when $\xi^{n-2i}=-1$) or a factor of $n^{\phi(n)}$.
For the latter case, recall that $\prod_{j=1}^{n-1}(1-\xi^j)=n$. 
Taking the algebraic norm of both sides, and noting that this norm is
completely multiplicative, gives us 
\begin{equation}\label{eqn:pp}
\prod_{j=1}^{n-1}N_n(1-\xi^j)=n^{\phi(n)}.
\end{equation} 
This suggests that each factor $N_n(1-\xi^j)$ on the left hand side of
Equation \eqref{eqn:pp} divides $n^{\phi(n)}$. But $\ell$ is not a
prime factor of $n$, and so $\ell$ cannot divide $\alpha$. Hence $H=G'$.
\end{proof}

\begin{lemma} \label{lem:lteiltnicht}
If $\ell \nmid 2^{\phi(n)}$ and $\ell \nmid n^{\phi(n)}$, then $K = T_{(q)}$.
\end{lemma}
\begin{proof}
Suppose there is $\id \ne g\in D_N$ which fixes the cosets, so
in particular $g((q)+1)=(q)+1$. This implies that $(q) \pm \xi^{i} =
(q) + 1$ for some integer $i$, and hence $1 \pm \xi^i \in (q)$. As in
the proof of the previous lemma, it follows then that $\ell \mid
N_n(1\pm\xi^i) = \beta$, where $\beta$ is either $2^{\phi(n)}$ or a
factor of $n^{\phi(n)}$. This is a contradiction, thus $K = T_{(q)}$. 
\end{proof}

\begin{table}
\caption{The cases $n=3,4,7,9$. Here, $j$ denotes the
  number of colourings with $\ell$ colours. Non-bracketed entries in
  the columns labelled $H$ and $K$ follow directly from results in
  this paper. Entries in brackets are computed by methods from
  \cite{pbeff}.}
\label{tab}
\begin{tabular}{|c|r|c|c|c|c|c|}
\hline\noalign{\smallskip}
$n$ & $\ell$ & $j$ & $H$ & $K$ & $q$ \\
\noalign{\smallskip}\hline\noalign{\smallskip}
3 & 3 & 1 & $G$ & $T_{(q)} \rtimes D_3$ & $1-\xi_3$ \\
  & 4 & 1 & $G$ & $T_{(q)} \rtimes C_2$ & $2$ \\
  & $> 4$ & * & * & $T_{(q)}$ & * \\
\hline
4 & 2 & 1 & $G$ & $T_{(q)} \rtimes D_4$ & $1-\xi_4$ \\
  & 4 & 1 & $G$ & $T_{(q)} \rtimes D_2$ & 2 \\
  & $> 4$ & * & * & $T_{(q)}$ & * \\
\hline
7 & 7 & 1 & $G$ & $T_{(q)} \rtimes D_7$ & $1- \xi_7$\\
  & 8 & 2 & \{$G'$\} & \{$T_{(q)} \rtimes C_2$\} & $1- \xi_7-\xi_7^3$\\
  & 29 & 6 & $G'$ & $T_{(q)}$ & $1- \xi_7 - \xi_7^2$\\
  & 43 & 6 & $G'$ & $T_{(q)}$ & $1- \xi_7 - \xi_7^2 - \xi_7^3$\\
  & 49 & 1 & $G$ & \{$T_{(q)}$\} & $(1- \xi_7)^2$\\
  & 56 & 2 & \{$G'$\} & $T_{(q)}$ & $(1- \xi_7)(1- \xi_7- \xi_7^3)$\\
  & 64 & 1 & $G$ & $T_{(q)} \rtimes C_2$,  & 2 \\
  & & 2 & \{$G'$\} & $T_{(q)}$  & $(1- \xi_7-\xi_7^3)^2$ \\
  & $> 64$ & * & * & $T_{(q)}$ & *\\
\hline
9 & 3 & 1 & $G$ & \{$T_{(q)} \rtimes D_9$\} & $1- \xi_9$\\
  & 9 & 1 & $G$ & \{$T_{(q)}$\} & $(1- \xi_9)^2$\\
  & 19 & 6 & $G'$ & $T_{(q)}$ & $1- \xi_9 - \xi_9^2$\\
  & 27 & 1 & $G$ & \{$T_{(q)}$\} & $1- \xi_9^3$\\
  & 37 & 6 & $G'$ & $T_{(q)}$ & $1- \xi_9 - \xi_9^3$\\
  & 57 & 6 & \{$G'$\} & $T_{(q)}$ & $(1- \xi_9)(1- \xi_9 - \xi_9^2)$\\ 
  & 64 & 1 & $G$ & $T_{(q)} \rtimes C_2$ & 2\\
  & $> 64$ & * & * & $T_{(q)}$  & *  \\
\noalign{\smallskip}\hline
\end{tabular}
\end{table}

Table \ref{tab} illustrates applications of the results of the last
two sections for the cases $n=3,4,7,9$. (These are exactly the values
of $n$ where $\phi(n) \in \{ 2,6 \}$. The cases $n=6, 14, 18$ are covered
implicitly.) This table can be seen as a complement to Table 4 in
\cite{pbeff}, where values of $n$ for which $\phi(n)=4$ are considered.
The entries in the second and third column follow from
\cite{bg}. Many entries in the fourth and fifth column follow
immediately from the results in this paper. Entries in brackets
require further computations (compare \cite{pbeff}), entries without
brackets are immediate.
Entries with an asterisk mean that there are multiple different
possibilities. The last column lists one (out of possibly more than
one) generator $q$ of a corresponding colouring. Note that for $n=7$,
there are three colourings with 64 colours. The colouring induced by
$(2)$ is perfect, while the other two colourings are not. 

\section{Application to quasiperiodic structures} \label{sec:5}

Consider a colouring $c$ of $\M_8$ with eight colours. There exists
exactly one such colouring \cite{bg}. By Corollary \ref{cor:one}, this
colouring is perfect. Therefore $H= \M_8 \rtimes D_8$.
Because of $N_8(1+\xi+\xi^2+\xi^3) = 8$, this colouring is defined by
the ideal $(q)=(1+\xi+\xi^2+\xi^3)$. Because of Theorem \ref{thm:bal}
and Lemma \ref{lem:bal}, $\overline{q}=1+\xi^7+\xi^6+\xi^5 \in
(q)$. Thus, $q+\overline{q}=2 \in (q)$. Hence the rotation by $\pi$
about $0$, which maps $1$ to $-1$, is contained in $K$.
The rotation by $\pi/2$ about 0 maps $1$ to $i$. The norm
of $1-i$ in $\M_8$ is $N_8(1-i)=4<8$, thus $1-i \notin (q)$. Therefore
$1$ and $i$ have different colours, and the rotation by $\pi/2$ and
thus the rotation by $\pi/4$ are not contained in $K$. Finally, the
reflection maps $\xi_8$ to $-\xi_8^3$. But $N_8(\xi_8+\xi_8^3)=4$
implying that $\xi_8$ and $-\xi_8^3$ have different colours. This
yields $K=T_{(q)} \rtimes{C_2}$ (where $C_2$ represents the rotation
by $\pi$).

Let us now describe how to illustrate this colouring $c$ and its
symmetries. Since $\M_8$ is dense in the plane, we want a
discrete subset of $\M_8$, which exhibits the colour symmetries
of $(\M_8,c)$. A colouring of such a subset is shown in Figure
\ref{fig:ab8col}. This set is well known from the theory of
aperiodic order: It is the vertex set of an Ammann Beenker tiling, see
\cite{gs} or \cite{sen2}. The symmetries discussed above are visible
in the image. 

Further examples of perfect and chirally perfect colourings of
quasiperiodic structures can be found in \cite{lip} (a 5-colouring of
the vertex set of the famous Penrose tiling, based on $\M_5$), in
\cite{bgs} (an 8-colouring of a quasiperiodic pattern based on
$\M_7$), in \cite{lueck} (several colourings based on $\M_n$ for
$n=4,6,8,10,12$),
and in \cite{pbeff} (a 4-colouring of the Ammann Beenker
tiling, based on $\M_8$). All these colourings arise from perfect
or chirally perfect colourings of $\M_n$.

\begin{figure}
\includegraphics[width=120mm,clip]{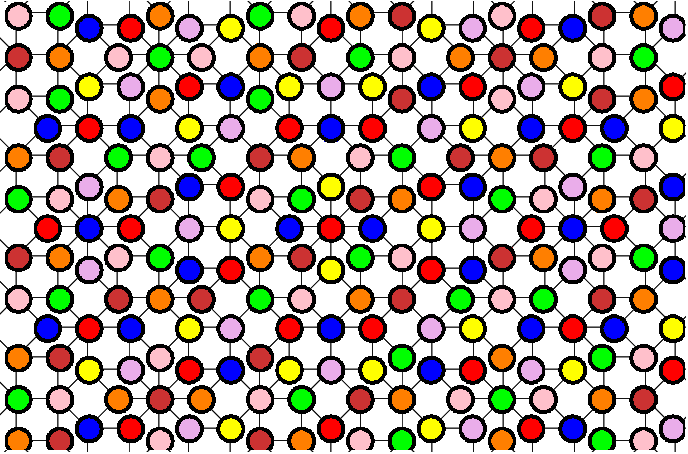}
\caption{\label{fig:ab8col} An 8-colouring of the vertices of the
  Ammann-Beenker tiling, arising from the 8-colouring of the
  underlying set $\M_8$.}
\end{figure}

\section{Conclusion} \label{sec:6}

Two classical special cases of colour symmetries are covered by our
approach, namely, the square lattice ($\M_4$) and the hexagonal
lattice ($\M_3$, resp.\ $\M_6$). These are discrete point sets. In
particular, we obtain Theorem 8.7.1 in \cite{gs} as a corollary, see
Corollary \ref{cor:z2}. All other cases ($n=5$, $n \ge 7$) yield point
sets $\M_n$ which are dense in the plane. In the case where $\M_n$ has
class number one, we obtained our main results. These are a necessary
and sufficient  condition for a colouring to be perfect (Theorem
\ref{thm:bal}). It allows the determination of the colour symmetry
group $H$ of $\M_n$ in general. In particular, it yields all perfect
colourings of $\M_n$. Moreover, for all but finitely
many cases, we determine the subgroup $K$ of $H$ of symmetries which
fix the coloured pattern: the lemmas in Section \ref{sec:4}
aid the derivation of $K$. A systematic way to
determine $K$ for a given ideal $\ell$-colouring would be
to check the conditions of Lemma \ref{lgleich2} ($\ell = 2$), Lemma
\ref{elllarge} ($\ell > 2^{\phi(n)}$), Lemma \ref{lem:q=2} and  Lemma
\ref{lem:normq=2n} ($\ell = 2^{\phi(n)}$), Lemma \ref{lem:lnprim}
($\ell = n$ prime), Lemma \ref{lem:lteiltnicht} ($\ell \nmid
2^{\phi(n)}$ and $\ell \nmid n^{\phi(n)}$). The remaining cases have
to be handled individually. This allows --- in principle --- to obtain
all colour preserving groups of (chirally) perfect colourings of
$\M_n$. 

For large $n$, the value $2^{\phi(n)}$ tends to be large, and it might be
tedious to handle the remaining cases individually. Nevertheless,
Lemma \ref{lem:lteiltnicht} seems to cover many of the remaining cases
of $\ell$, compare Table 5 in \cite{bg}. For instance, for $n=15$,
there are 11 cases for which $\ell \leq 2^{\phi(15)}=256$. Our results
cover 8 out of 11 cases, only three cases require further effort in
order to derive the colour preserving group $K$. To give another
example, for $n=16$, there are 23 cases for $\ell \leq
2^{\phi(16)}=256$, but 17 of them are covered by our results, and only
six cases have to be checked individually in order to determine the
group $K$.

\section*{Acknowledgements}
The authors are grateful to Michael Baake and Christian Huck for
helpful discussions. They wish to express their thanks to the
CRC 701 of the German Research Council (DFG). The research leading 
to these results has received funding from the European
Research Council under the European Union's Seventh Framework 
Programme (FP7/2007-2013) / ERC grant agreement no 247029.

\end{document}